\documentclass{amsart}
\usepackage{latexsym,amssymb,amsmath}

\theoremstyle{plain}
\newtheorem{lemma}{Lemma}[section]
\newtheorem{theorem}[lemma]{Theorem}
\newtheorem{corollary}[lemma]{Corollary}
\newtheorem{algorithm}[lemma]{Algorithm}

\theoremstyle{definition}

\newtheorem{example}[lemma]{Example}
\newtheorem{counterexample}[lemma]{Counterexample}
\newtheorem{remark}[lemma]{Remark}

\begin{document}


\title[ ]{How to compute the Hilbert depth of a graded ideal}
\author[ ]{Ri-Xiang Chen}
\address{Department of Mathematics, Nanjing University of Science and Technology, Nanjing, Jiangsu, 210094, P.R.China }
\email{rc429@cornell.edu }
\date{}
\begin{abstract}
We give two algorithms for computing the Hilbert depth of a \emph{graded ideal} in the polynomial ring. 
These algorithms work efficiently for (squarefree) lex ideals. As a consequence, we  construct counterexamples to some conjectures made by Shen in \cite{B:Sh2}.
\end{abstract}
\maketitle

\section{Introduction}

\noindent Let $S=K[x_1, \ldots, x_n]$ be the polynomial 
ring over a field $K$ with the standard $\mathbb{Z}^n$-grading. 
Let $I$ be a $\mathbb{Z}^n$-graded ideal in $S$, then $I$ 
is a monomial ideal. A \emph{Stanley decomposition} of $I$ 
is a direct sum
\[
\mathcal{D}: I=\bigoplus_{i=1}^{m}m_iK[Z_i]
\]
as $K$-vector spaces, where $m_i\in I$ is a monomial and 
$Z_i\subset \{x_1,\ldots,x_n\}$. We define the \emph{Stanley depth} of $\mathcal{D}$ as
\[
\mbox{sdepth}(\mathcal{D}):=\min\{|Z_i|\mid 1\leq i \leq m\},
\]
and the \emph{Stanley depth} of $I$ as
\[
\mbox{sdepth}(I):=\max\{\mbox{sdepth}(\mathcal{D})\mid \mathcal{D} \ \mbox{is a Stanley decomposition of}\  I\}.
\]
Stanley's Conjecture \cite{B:St} says that $\mbox{sdepth}(I)\geq \mbox{depth}(I)$. 
Except some special cases, this conjecture remains open.(For details, see the introduction in \cite{B:Sh2}.)

In general, it is hard to compute $\mbox{sdepth}(I)$. A breakthrough 
was made by Herzog, Vladoiu and Zheng in \cite{B:HVZ}, where the 
computation of $\mbox{sdepth}(I)$ was converted to the problem of 
patitions of the poset $P_I^g$ into intervals.(see \cite{B:HVZ} section 2 for details.) 
With this method, many results were obtained. For example,
Bir\'{o} et al. \cite{B:BHK} showed 
\[
\mbox{sdepth}(\mathfrak{m})=\lceil \frac{n}{2} \rceil,
\]
where $\mathfrak{m}=(x_1, \ldots, x_n)\subset S$; 
Keller et al. \cite{B:KSS} showed
\[
\mbox{sdepth}(I_{n,d})=d+\lfloor \frac{n-d}{d+1} \rfloor \  \ \ \mbox{for}\ 1\leq d \leq n < 5d+4,
\]
where $I_{n,d}$ is the squarefree Veronese ideal 
generated by all degree $d$ squarefree monomials in $S$.
Since the poset $P_I^g$ often contains many elements, the 
partitions of $P_I^g$ into intervals can be very complicated. 
Because of this, we still do not know if $\mbox{sdepth}(I_{n,d})=d+\lfloor \frac{n-d}{d+1} \rfloor $
holds for all $d\leq n$. 

Closely related to Stanley depth is the concept of 
Hilbert depth, which was introduced by Bruns et al. in \cite{B:BKU1}.
Now suppose that $S=K[x_1,\ldots, x_n]$ is $\mathbb{Z}$-graded 
with $\deg(x_i)=1$ for $1\leq i \leq n$.
Let $I$ be a $\mathbb{Z}$-graded ideal in $S$. 
A \emph{Hilbert decomposition} of $I$ is an isomorphism 
\[
\mathcal{H}: I\cong \bigoplus_{i=1}^{m}K[Z_i](-s_i) \tag{\dag}
\]
as $K$-vector spaces, where $s_i\in \mathbb{Z}_{\geq 0}$ 
and $Z_i\subset \{x_1,\ldots,x_n\}$. 
We define the \emph{Hilbert depth} of $\mathcal{H}$ as
\[
\mbox{hdepth}(\mathcal{H}):=\min\{|Z_i|\mid 1\leq i \leq m\},
\]
and the \emph{Hilbert depth} of $I$ as
\[
\mbox{hdepth}(I):=\max\{\mbox{hdepth}(\mathcal{H})\mid \mathcal{H} \ \mbox{is a Hilbert decomposition of}\  I\}.
\]
Note that for simplicity, $\mbox{hdepth}(I)$ in this paper 
is the same as $\mbox{Hdepth}_1(I)$ in \cite{B:BKU1} and 
$\mbox{hdepth}_1(I)$ in \cite{B:Sh2}. Also, in the rest of 
this paper, by a graded ideal, we mean a $\mathbb{Z}$-graded ideal.

For a monomial ideal $I$ in $S$, since a Stanley 
decomposition 
\[
\mathcal{D}: I=\bigoplus_{i=1}^{m}m_iK[Z_i]
\]
induces a Hilbert decomposition
\[
\mathcal{H}: I\cong \bigoplus_{i=1}^{m}K[Z_i](-\deg(m_i)),
\]
it follows that $\mbox{sdepth}(I)\leq \mbox{hdepth}(I)$.
In other words, Hilbert decomposition is weaker than
Stanley decompostion and $\mbox{hdepth}(I)$ gives an upper 
bound for $\mbox{sdepth}(I)$.

In general, $\mbox{hdepth}(I)$ is easier to compute than 
$\mbox{sdepth}(I)$, because $\mbox{hdepth}(I)$ depends only 
on the Hilbert series of $I$. A Hilbert decomposition ($\dag$)
is equivalent to a decomposition of the Hilbert series $H_I(t)$:
\[
H_I(t)=\sum_{i=1}^m\frac{t^{s_i}}{(1-t)^{|Z_i|}}.
\]
Actually, we have the following key theorem about $\mbox{hdepth}(I)$:

\begin{theorem}[Uliczka \cite{B:Ul}] \label{T:Ul}
Let $I$ be a graded ideal in the polynomial ring 
$S=K[x_1,\ldots, x_n]$, then $\mbox{hdepth}(I)$ is the same as:
\begin{itemize}
\item[(1)] the maximal $p$ such that all the coefficients in
           the power series $(1-t)^pH_I(t)$ are non-negative;
\item[(2)] the maximal $p$ such that $H_I(t)$ can be written as 
           \[
           H_I(t)=\sum_{i=p}^n\frac{Q_i(t)}{(1-t)^i},
           \]
           where $Q_i(t)\in \mathbb{Z}_{\geq 0}[t]$.
\end{itemize}
\end{theorem}

With this tool, Bruns et al. \cite{B:BKU2} showed 
\[
\mbox{hdepth}(\mathfrak{m}^d)=\lceil \frac{n}{d+1} \rceil,
\]
where $\mathfrak{m}=(x_1,\ldots, x_n)\subset S$ and $d\geq 1$;
Ge et al. \cite{B:GLW} showed
\[
\mbox{hdepth}(I_{n,d})=d+\lfloor \frac{n-d}{d+1} \rfloor,
\]
where $I_{n,d}$ is the squarefree Veronese ideal. As pointed out 
in section 4 of \cite{B:Sh2}, these two results are equivalent, 
since $H_{\mathfrak{m}^d}(t)=(1-t)^{d-1}H_{I_{n+d-1,d}}(t)$.
By comparing Hilbert depth and Stanley depth, it is natural 
to ask if $\mbox{sdepth}(\mathfrak{m}^d)=\lceil \frac{n}{d+1} \rceil$ holds for $d\geq 2$.
Little is known about this.

Hilbert depth can help us understand Stanley depth. Conversely,
results about Stanley depth can shed some light on Hilbert depth too.
Section 2 in this paper is actually inspired by section 2 in \cite{B:HVZ}. 
Theorem \ref{T:key}  here about Hilbert depth is analogous to 
Theorem 2.1 and Theorem 2.4 in \cite{B:HVZ}, which 
are about Stanley depth. By Theorem \ref{T:key}, we will develop two algorithms 
(Algorithm \ref{A:algorithm1} and Algorithm \ref{A:algorithm2}) 
for computing the Hilbert depth of a graded ideal. And 
many interesting examples are computed in section 2. 

In \cite{B:Sh2}, some conjectures were made about the 
Hilbert depth and Stanley depth of a lex ideal generated by 
monomials of the same degree. In Section 3 of this paper we 
will give some counterexamples to these conjectures. 
And Algorithm \ref{A:algorithm2} will be used in the computations 
of these counterexamples. 

The algorithms in this paper work only for \emph{graded ideals} in
the polynomial ring. They are different form the algorithm given by Popescu
in \cite{B:Po} and the algorithm given by Bruns et al. in \cite{B:BMU}. 
Their algorithms work for \emph{all modules} over the polynomial ring. 
The differences among these algorithms will be illustrated in 
Examples \ref{E:ex2},  \ref{E:ex3}, \ref{E:r1}, \ref{E:r2} and Remark \ref{R:comparison}.

The author was originally interested in finding some counterexamples to 
the conjectures in \cite{B:Sh2} and in using the Hilbert depth to help study 
the Stanley depth of squarefree Veronese ideals. After computing 
many examples, Algorithm \ref{A:algorithm2} was first developed. 
In order to give the algorithm a proof, Theorem \ref{T:key} and Algorithm 
\ref{A:algorithm1} were then found. It turns out that these algorithms
are useful, especially for computing the Hilbert depth of (squarefree) lex ideals.

\section{Algorithms for Computing the Hilbert Depth}
\noindent This section is inspired by section 2 in \cite{B:HVZ}. In the following 
theorem, $f(t)$ is analogous to $P_I^g$; a decomposition of $f(t)$ as in ($\dag\dag$)
is analogous to a partition of $P_I^g$ into intervals; 
part (2) is analogous to Theorem 2.1 in \cite{B:HVZ}; 
part (3) is analogous to Theorem 2.4 in \cite{B:HVZ}. 

\begin{theorem}\label{T:key}
Let $I$ be a proper squarefree monomial ideal in  $S=K[x_1,\ldots, x_n]$.
Let $I$ be minimally generated by monomials $u_1, \ldots, u_s$ with
$1\leq d=\deg(u_1)\leq \cdots \leq \deg(u_s)$. For all $ d\leq i \leq n$, 
let $a_i$ be the number of squarefree monomials of degree $i$ in $I$.
Set
\[
f(t):=a_dt^d+\cdots+a_nt^n=\sum_{i=d}^na_it^i.
\]
Then
\begin{itemize}
\item[(1)] $H_I(t)=\frac{a_dt^d}{(1-t)^d}+\cdots+\frac{a_nt^n}{(1-t)^n}=\displaystyle \sum_{i=d}^n\frac{a_it^i}{(1-t)^i}$ and $\mbox{hdepth}(I)\geq d$;
\item[(2)] if 
           \[
           f(t)=\sum_jb_jt^{\alpha_j}(1+t)^{\beta_j}, \  \mbox{with} \ b_j\in \mathbb{Z}_{>0}, \alpha_j\geq d, \beta_j\in \mathbb{Z}_{\geq 0}, \tag{\dag\dag}
           \]
           then
           \[
           H_I(t)=\sum_j\frac{b_jt^{\alpha_j}}{(1-t)^{\alpha_j+\beta_j}} \ \mbox{and} \ \mbox{hdepth}(I)\geq \min_j\{\alpha_j+\beta_j\};
           \]
\item[(3)] if $\mbox{hdepth}(I)\geq p$ with $d\leq p\leq n$, then
           there exist unique $b_d,\ldots, b_p\in \mathbb{Z}_{\geq 0}$ such that
           \[
           f(t)=\left( \sum_{i=d}^pb_it^i(1+t)^{p-i} \right) +a_{p+1}t^{p+1}+\cdots+a_nt^n,
           \]         
           and then by part (2),
           \[
           H_I(t)=\frac{b_dt^d+b_{d+1}t^{d+1}+\cdots+b_pt^p}{(1-t)^p}+\frac{a_{p+1}t^{p+1}}{(1-t)^{p+1}}+\cdots+\frac{a_nt^n}{(1-t)^n}.
           \]           
\end{itemize}
\end{theorem}

\begin{proof}
(1)Let $g=(1,\ldots, 1)$, then by the method of \cite{B:HVZ}, the trivial
partition of $P_I^g$, 
\[
P_I^g=\bigcup_{u\in P_I^g}[u,u],
\]
induces a Stanley decomposition of $I$. Suppose $\deg(u)=l$ and $u=x_{i_1}\cdots x_{i_l}$, then
\[
I=\bigoplus_{u\in P_I^g} uK[x_{i_1},\ldots, x_{i_l}].
\]
Thus, the multigraded Hilbert series of $I$ is 
\[
H_I(T_1,\ldots,T_n)=\sum_{u\in P_I^g}\frac{T_{i_1}\cdots T_{i_l}}{(1-T_{i_1})\cdots(1-T_{i_l})}.
\]
So, the Hilbert series of $I$ is
\[
H_I(t)=\sum_{u\in P_I^g}\frac{t^{\deg(u)}}{(1-t)^{\deg(u)}}=\sum_{i=d}^n\frac{a_it^i}{(1-t)^i}.
\]

(2)From the proof of part (1), we see that a $t^i$ in $f(t)$ corresponds to a squarefree 
monomial of degree $i$ in $I$, which gives rise to a term $\frac{t^i}{(1-t)^i}$ in $H_I(t)$.
Thus, 
\[
t^{\alpha_i}(1+t)^{\beta_i}=\binom{\beta_j}{0}t^{\alpha_j}+\binom{\beta_j}{1}t^{\alpha_j+1}+\cdots+\binom{\beta_j}{\beta_j}t^{\alpha_j+\beta_j}
\]
gives rise to 
\[
\binom{\beta_j}{0}\frac{t^{\alpha_j}}{(1-t)^{\alpha_j}}+\binom{\beta_j}{1}\frac{t^{\alpha_j+1}}{(1-t)^{\alpha_j+1}}
+\cdots+\binom{\beta_j}{\beta_j}\frac{t^{\alpha_j+\beta_j}}{(1-t)^{\alpha_j+\beta_j}}
\]
in $H_I(t)$. The latter is equal to 
\begin{align*}
  & \frac{\binom{\beta_j}{0}t^{\alpha_j}(1-t)^{\beta_j}+\binom{\beta_j}{1}t^{\alpha_j+1}(1-t)^{\beta_j-1}
  +\cdots+\binom{\beta_j}{\beta_j}t^{\alpha_j+\beta_j}(1-t)^{0}}{(1-t)^{\alpha_j+\beta_j}}\\
  &= \frac{t^{\alpha_j}\left(\binom{\beta_j}{0}(1-t)^{\beta_j}+\binom{\beta_j}{1}t^{1}(1-t)^{\beta_j-1}
  +\cdots+\binom{\beta_j}{\beta_j}t^{\beta_j}(1-t)^{0}\right)}{(1-t)^{\alpha_j+\beta_j}}\\
  &=\frac{t^{\alpha_j}\left(t+(1-t)\right)^{\beta_j}}{(1-t)^{\alpha_j+\beta_j}}\\
  &=\frac{t^{\alpha_j}}{(1-t)^{\alpha_j+\beta_j}}.
\end{align*}
Therefore, if 
\[
           f(t)=\sum_jb_jt^{\alpha_j}(1+t)^{\beta_j}, 
\]
then
\[
H_I(t)=\sum_j\frac{b_jt^{\alpha_j}}{(1-t)^{\alpha_j+\beta_j}}.
\]
And by Theorem \ref{T:Ul}, we have 
\[
\mbox{hdepth}(I)\geq \min_j\{\alpha_j+\beta_j\}.
\]

(3)Note that
\[
(1-t)^pH_I(t)=\left((1-t)^{\mbox{hdepth}(I)}H_I(t)\right)\cdot\frac{1}{(1-t)^{\mbox{hdepth}(I)-p}}.
\]
By Theorem \ref{T:Ul}, all the coefficients in the power series $(1-t)^{\mbox{hdepth}(I)}H_I(t)$ are non-negative.
By the assumption, $\mbox{hdepth}(I)\geq p$, it follows that all 
the coefficients in the power series $\frac{1}{(1-t)^{\mbox{hdepth}(I)-p}}$ are non-negative.
Thus, all the coefficients in the power series $(1-t)^pH_I(t)$ are non-negative.

By part (1), we have
\begin{align*}
  & (1-t)^pH_I(t)\\
  &= (1-t)^p\left(\frac{a_dt^d}{(1-t)^d}+\frac{a_{d+1}t^{d+1}}{(1-t)^{d+1}}+\cdots+\frac{a_nt^n}{(1-t)^n}\right)\\
  &=a_dt^d(1-t)^{p-d}+a_{d+1}t^{d+1}(1-t)^{p-d-1}+\cdots+a_{p-1}t^{p-1}(1-t)^1+a_pt^p\\
  &\ \ +\frac{a_{p+1}t^{p+1}}{(1-t)^{p+1}}+\cdots+\frac{a_nt^n}{(1-t)^n}.
\end{align*}
Since in the above expression all the terms in the first $p-d+1$ parts are of degree $\leq p$
and all the terms in the last $n-p$ parts are of degree $\geq p+1$, it follows that all the coefficients 
in the degree $p$ polynomail 
\[
a_dt^d(1-t)^{p-d}+a_{d+1}t^{d+1}(1-t)^{p-d-1}+\cdots+a_{p-1}t^{p-1}(1-t)^1+a_pt^p
\]
are non-negative. $\forall \ d\leq i \leq p$, the coefficient of $t^i$ in the above polynomial is 
\[
a_i-\binom{p-(i-1)}{1}a_{i-1}+\binom{p-(i-2)}{2}a_{i-2}-\cdots+(-1)^{i-d}\binom{p-d}{i-d}a_d.
\]
$\forall \ d\leq i \leq p$, set
\[
b_i:=a_i-\binom{p-(i-1)}{1}a_{i-1}+\binom{p-(i-2)}{2}a_{i-2}-\cdots+(-1)^{i-d}\binom{p-d}{i-d}a_d, \tag{*}
\]
then $b_d,\ldots, b_p\in \mathbb{Z}_{\geq 0}$.

Claim: with $b_d, \ldots, b_p$ defined as above, the following identity holds:
\[
b_dt^d(1+t)^{p-d}+b_{d+1}t^{d+1}(1+t)^{p-d-1}+\cdots+b_pt^p=a_dt^d+a_{d+1}t^{d+1}+\cdots+a_pt^p. \tag{**}
\]

Indeed, suppose
\[
b_dt^d(1+t)^{p-d}+b_{d+1}t^{d+1}(1+t)^{p-d-1}+\cdots+b_pt^p=c_dt^d+c_{d+1}t^{d+1}+\cdots+c_pt^p.
\]
Let
\begin{align*}
A&= 
  \begin{pmatrix} \binom{p-d}{0} & 0 & 0& \cdots & 0\\ \binom{p-d}{1} & \binom{p-d-1}{0} & 0& \cdots & 0\\ \binom{p-d}{2} & \binom{p-d-1}{1} & \binom{p-d-2}{0}& \cdots & 0\\ \vdots & \vdots & \vdots & \ddots & \vdots \\\binom{p-d}{p-d} & \binom{p-d-1}{p-d-1} & \binom{p-d-2}{p-d-2}& \cdots & 1\end{pmatrix}, \\ 
B&=\begin{pmatrix} \binom{p-d}{0} & 0 & 0& \cdots & 0\\ -\binom{p-d}{1} & \binom{p-d-1}{0} & 0& \cdots & 0\\ \binom{p-d}{2} & -\binom{p-d-1}{1} & \binom{p-d-2}{0}& \cdots & 0\\ \vdots & \vdots & \vdots & \ddots & \vdots \\(-1)^{p-d}\binom{p-d}{p-d} & (-1)^{p-d-1}\binom{p-d-1}{p-d-1} & (-1)^{p-d-2}\binom{p-d-2}{p-d-2}& \cdots & 1\end{pmatrix} .
\end{align*}
Then we have
\[
  \begin{pmatrix} c_d\\ c_{d+1}\\ c_{d+2}\\ \vdots \\c_p \end{pmatrix}= 
A\begin{pmatrix} b_d\\ b_{d+1}\\ b_{d+2}\\ \vdots \\b_p \end{pmatrix}=
AB \begin{pmatrix} a_d\\ a_{d+1}\\ a_{d+2}\\ \vdots \\a_p \end{pmatrix}.
\]
Let $1\leq j<i \leq p-d+1$, then the $i$-th row of the matrix $A$ is 
\[
\begin{pmatrix} \binom{p-d}{i-1} & \binom{p-d-1}{i-2}  & \cdots &\binom{p-d-i+1}{0} & 0 & \cdots & 0 \end{pmatrix},
\]
and the transpose of the $j$-th column of the matrix $B$ is 
\[
\begin{pmatrix} 0& \cdots & 0 & \binom{p-d-j+1}{0} & -\binom{p-d-j+1}{1} & \cdots & (-1)^{p-d-j+1}\binom{p-d-j+1}{p-d-j+1}\end{pmatrix}.
\]
Their inner product is 
\begin{align*}
&\begin{smallmatrix} \binom{p-d-j+1}{0}\binom{p-d-j+1}{i-j}- \binom{p-d-j+1}{1}\binom{p-d-j}{i-j-1}
+\binom{p-d-j+1}{2}\binom{p-d-j-1}{i-j-2}-\cdots +(-1)^{i-j}\binom{p-d-j+1}{i-j}\binom{p-d-i+1}{0} \end{smallmatrix}\\ 
&=\begin{matrix} \frac{(p-d-j+1)!}{(p-d-i+1)!}\left( \frac{1}{0!(i-j)!}-\frac{1}{1!(i-j-1)!}+\frac{1}{2!(i-j-2)!}-\cdots +(-1)^{i-j}\frac{1}{(i-j)!0!}\right) \end{matrix}\\
&=\begin{matrix} \frac{(p-d-j+1)!}{(p-d-i+1)!(i-j)!}\left( \frac{(i-j)!}{0!(i-j)!}-\frac{(i-j)!}{1!(i-j-1)!}+\frac{(i-j)!}{2!(i-j-2)!}-\cdots +(-1)^{i-j}\frac{(i-j)!}{(i-j)!0!}\right) \end{matrix}\\
&=\begin{matrix} \frac{(p-d-j+1)!}{(p-d-i+1)!(i-j)!}\left( \binom{i-j}{0}-\binom{i-j}{1}+\binom{i-j}{2}-\cdots +(-1)^{i-j}\binom{i-j}{i-j}\right) \end{matrix}\\
&=\begin{matrix} \frac{(p-d-j+1)!}{(p-d-i+1)!(i-j)!}(1-1)^{i-j} \end{matrix}\\
&=0.
\end{align*}
Since $AB$ is also a lower triangular matrix with  diagonal elements equal to $1$, 
it follows that $AB$ is the identity matrix. Therefore, 
$c_d=a_d, \ldots, c_p=a_p$ and the claim is proved. 

So we have found  $b_d,\ldots, b_p\in \mathbb{Z}_{\geq 0}$ such that
\[
f(t)=\left( \sum_{i=d}^pb_it^i(1+t)^{p-i} \right) +a_{p+1}t^{p+1}+\cdots+a_nt^n.
\]
From the proof of the claim, we see that $b_d,\ldots, b_p$ are 
uniquely determined by $a_d,\ldots, a_p$ because the matrix $B$ is invertible. 
\end{proof}
 
Note that the identity (**) in the above proof holds even if 
$a_d, \ldots, a_p, b_d, \ldots, b_p$ are real numbers. 
Instead of using the formula $(*)$, we can easily
compute $b_d, \ldots, b_p$ form $a_d, \ldots, a_p$ by 
the following simple algorithm. 

\begin{algorithm}\label{A:d1}
Let $g(t)=a_dt^d+\cdots+a_pt^p$ be a polynomial in $t$ with $a_d,\ldots, a_p\in \mathbb{R}$. 
\noindent Input: $g(t)=a_dt^d+\cdots+a_pt^p$.

\noindent Step $1$: Let $b_d=a_d$ and let $g_1(t)=g(t)-b_dt^d(1+t)^{p-d}$. Simplify $g_1(t)$ to 
                  get $g_1(t)=a_{d+1}^{(1)}t^{d+1}+\cdots+a_{p}^{(1)}t^p$ with $a_{d+1}^{(1)},\ldots, a_p^{(1)}\in \mathbb{R}$.
                  
\noindent Step $2$: Let $b_{d+1}=a_{d+1}^{(1)}$ and let $g_2(t)=g_1(t)-b_{d+1}t^{d+1}(1+t)^{p-d-1}$. Simplify $g_2(t)$ to 
                  get $g_2(t)=a_{d+2}^{(2)}t^{d+2}+\cdots+a_{p}^{(2)}t^p$ with $a_{d+2}^{(2)},\ldots, a_p^{(2)}\in \mathbb{R}$.   

\noindent \ldots \ldots \ldots 

\noindent Step $(p-d)$: Let $b_{p-1}=a_{p-1}^{(p-d-1)}$ and let $g_{p-d}(t)=g_{p-d-1}(t)-b_{p-1}t^{p-1}(1+t)$. Simplify $g_{p-d}(t)$ to 
                  get $g_{p-d}(t)=a_{p}^{(p-d)}t^p$ with $a_p^{(p-d)}\in \mathbb{R}$.     
                  
\noindent Step $(p-d+1)$: Let $b_p=a_{p}^{(p-d)}$.

\noindent Output: $b_d,\ldots, b_p\in \mathbb{R}$.                           
\end{algorithm}

By Theorem \ref{T:key}, we immediately have the following corollary, 
which follows the style of Theorem \ref{T:Ul}.

\begin{corollary}\label{C:key}
Let $I$ be a proper squarefree monomial ideal in $S$.
With the notations as in Theorem \ref{T:key}, we have that 
$\mbox{hdepth}(I)$ is equal to the maximal $p$ such that 
\[
f(t)=\left( \sum_{i=d}^pb_it^i(1+t)^{p-i} \right)+a_{p+1}t^{p+1}+\cdots+a_nt^n,
\]
with $b_d,\ldots, b_p\in \mathbb{Z}_{\geq 0}$. In particular, 
$\mbox{hdepth}(I)$ can be calculated in a finite number of steps. 
\end{corollary}

By the above results, when $I$ is a squarefree monomial ideal, 
we have the following algorithm to compute $\mbox{hdepth}(I)$.

\begin{algorithm}\label{A:squarefree}
Let $I$ be a proper squarefree monomial ideal in  $S=K[x_1,\ldots, x_n]$.
Let $I$ be minimallly generated by monomials $u_1, \ldots, u_s$ with
$1\leq d=\deg(u_1)\leq \cdots \leq \deg(u_s)$. 

\noindent Input: $I$.

\noindent Step 1: 
$\forall\  d\leq i \leq n$, 
count the number of squarefree monomials of degree $i$ in $I$ and denote it by $a_i$.
Set 
\[
f(t):=a_dt^d+\cdots+a_nt^n, \ \mbox{and} \ j:=\min\{\lfloor  \frac{a_{d+1}}{a_d} \rfloor, n-d\}.
\]

\noindent Step 2: Let $g(t):=a_dt^d+\cdots+a_{d+j}t^{d+j}$. Run algorithm \ref{A:d1} for $g(t)$ 
to get $b_d, \ldots, b_{d+j}$. If $b_d, \ldots, b_{d+j}\in \mathbb{Z}_{\geq 0}$ then put 
$\mbox{hdepth}(I)=d+j$ and stop; otherwise, set $j:=j-1$ and run step 2 again.

\noindent Output: $\mbox{hdepth}(I)$.

\end{algorithm}

\begin{remark}
Since $\mbox{hdepth}(I)\leq n$, in step 1 we choose the initial $j\leq n-d$. 
Also, if $j\geq \lfloor  \frac{a_{d+1}}{a_d} \rfloor+1$ then in step 2 we will
have $b_{d+1}<0$, so that we choose the initial $j\leq \lfloor  \frac{a_{d+1}}{a_d} \rfloor$. 
Finally, it is easy to see that step 2 stops in a finite number of steps. 
\end{remark}

Next we turn our attention to graded ideals in the polynomial ring. 
As to Hilbert functions, we have the classical Macaulay's Theorem.

\begin{theorem}[Macaulay \cite{B:Ma}]\label{T:Ma}
Let $I$ be a graded ideal in the polynomial ring $S=K[x_1,\ldots, x_n]$, 
then there exists a lex ideal $L$ in $S$ such that $H_L(t)=H_I(t)$.
\end{theorem}

Since Hilbert depth depends only on the Hilbert series, 
it follows that 
\[
\mbox{hdepth}(I)=\mbox{hdepth}(L).
\]

Due to \cite{B:AHH}, a lex ideal $L$ can be related to a 
squarefree strongly stable monomial ideal $L^\sigma$, 
where the so-called \emph{squarefree operator} $\sigma$ was introduced by Gil Kalai. 
Let $u=x_{i_1}x_{i_2}\cdots x_{i_d}\in S$ with $i_1\leq i_2\leq \cdots \leq i_d$.
We define $\mbox{m}(u):=i_d$ and $u^\sigma:=x_{i_1}x_{i_2+1}\cdots x_{i_d+d-1}$,
which is squarefree. 
If $I$ is a monomial ideal in $S$ minimally generated by monomials $u_1, \ldots, u_s$,
then we define $I^\sigma$ to be the squarefree monomial ideal generated by 
$u_1^\sigma, \ldots, u_s^\sigma$ in $S'=K[x_1,\ldots,x_m]$, 
where $m=\max\{\mbox{m}(u_i)+\deg(u_i)-1\mid 1\leq i \leq s \}$.

\begin{lemma}[\cite{B:AHH} 1.2]\label{L:AHH}
Let $I$ be a strongly stable monomial ideal in $S=K[x_1,\ldots,x_n]$
minimally generated by $u_1, \ldots, u_s$, then $I^\sigma$ is a squarefree 
strongly stable monomial ideal in $S'=K[x_1, \ldots, x_m]$ minimally 
generated by $u_1^\sigma, \ldots, u_s^\sigma$ with $m=\max\{\mbox{m}(u_i)+\deg(u_i)-1\mid 1\leq i \leq s \}$.
\end{lemma}

\begin{remark}
Lemma 1.8 in \cite{B:AHH} says that if $L$ is a lex ideal in $S$ then 
$L^\sigma$ is a squarefree lex ideal in $S'$. 
In general, this is \emph{not true} as we will see in Example \ref{E:ex1}. 
However, if $L$ is a lex ideal generated by monomials in the same degree, 
then $L^\sigma$ is squarefree lex, because there is an order-preserving 
bijection between the set of degree $d$ monomials in $K[x_1,\ldots,x_n]$ 
and the set of degree $d$ squarefree monomials in $K[x_1,\ldots, x_{n+d-1}]$. 
Anyway, when $L$ is a lex ideal, by the above lemma, we can conclude that 
$L^\sigma$ is a squarefree strongly stable monomial ideal. 
\end{remark}

\begin{remark}\label{R:relation}
By Lemma 2.2 in \cite{B:AHH}, if $L$ is a lex ideal then $L$ and $L^\sigma$ 
have the same graded Betti numbers, that is, $\beta_{i,j}^S(L)=\beta_{i,j}^{S'}(L^\sigma)$.
Since
\[
H_L(t)=\frac{\sum_{i,j}(-1)^i\beta_{i,j}^S(L)t^j}{(1-t)^n}
\]
and 
\[
H_{L^\sigma}(t)=\frac{\sum_{i,j}(-1)^i\beta_{i,j}^{S'}(L^\sigma)t^j}{(1-t)^m},
\]
it follows that 
\[
H_{L^\sigma}(t)=H_L(t)\frac{1}{(1-t)^{m-n}}.
\]
So, by Theorem \ref{T:Ul}, if $\mbox{hdepth}(L^\sigma)=p$ then $\mbox{hdepth}(L)=p-(m-n)$.
\end{remark}

By the above results, we have the following algorithm 
for computing the Hilbert depth of a graded ideal.

\begin{algorithm}\label{A:algorithm1}
Let $I$ be a graded ideal in the polynomial ring $S=K[x_1,\ldots,x_n]$.

\noindent Input: $I$

\noindent Step 1: Use Grobner basis theory to get the initial ideal $\mbox{in}_{<\mbox{lex}}(I)$.
                  Let $L$ be the lex ideal in $S$ with the same Hilbert function as $\mbox{in}_{<\mbox{lex}}(I)$.
                  Find the minimal generators of $L$ and denote them by $u_1,\ldots, u_s$.
                  
\noindent Step 2: Let $m=\max\{\mbox{m}(u_i)+\deg(u_i)-1\mid 1\leq i \leq s \}$.
                  Let $L^\sigma=(u_1^\sigma, \ldots, u_s^\sigma)\subset S'=K[x_1,\ldots, x_m]$.
                  Use Algorithm \ref{A:squarefree} to compute $\mbox{hdepth}(L^\sigma)$.
                  
\noindent Output: $\mbox{hdepth}(L)=\mbox{hdepth}(L^\sigma)-(m-n)$.                                    
\end{algorithm}

\begin{example}\label{E:ex1}
Let $I=(x_1^2, x_1x_2, x_1x_3, x_2^2, x_3^2)\subset S=K[x_1,x_2,x_3]$, then
\[
L=(x_1^2, x_1x_2, x_1x_3, x_2^2, x_2x_3, x_3^3)
\]
is the lex ideal in $S$ with the same Hilbert function as $I$, and 
\[
L^\sigma=(x_1x_2, x_1x_3, x_1x_4, x_2x_3, x_2x_4, x_3x_4x_5)\subset S'=K[x_1, x_2, x_3, x_4, x_5].
\]
Note that $L^\sigma$ is a squarefree strongly stable monomial ideal in $S'$, but 
$L^\sigma$ is not squarefree lex, because $x_1x_5>_{\mbox{lex}}x_2x_3$ and $x_1x_5\notin L^\sigma$.

In $L^\sigma$, the squarefree monomials of degree $2$ are $x_1x_2, x_1x_3, x_1x_4, x_2x_3, x_2x_4$;
the squarefree monomials of degree $3$ are $x_1x_2x_3, x_1x_2x_4,x_1x_2x_5,x_1x_3x_4,x_1x_3x_5$,
$x_1x_4x_5,x_2x_3x_4,x_2x_3x_5,x_2x_4x_5,x_3x_4x_5$;
the squarefree monomials of degree $4$ are $x_1x_2x_3x_4,x_1x_2x_3x_5,x_1x_2x_4x_5,x_1x_3x_4x_5,x_2x_3x_4x_5$;
the squarefree monomials of degree $5$ is $x_1x_2x_3x_4x_5$. Thus,
\[
f(t)=5t^2+10t^3+5t^4+t^5.
\]
Let $j=\lfloor \frac{10}{5} \rfloor=2$ then $f(t)=5t^2(1+t)^2+t^5$. 
So, $\mbox{hdepth}(L^\sigma)=2+2=4$ and then
$\mbox{hepth}(I)=\mbox{hepth}(L)=4-(5-3)=2$.
\end{example}

\begin{remark}\label{R:remarkex1}
By the method of \cite{B:HVZ} we see that $\mbox{sdepth}(L^\sigma)=3<\mbox{hdepth}(L^\sigma)$.
Indeed, we can think of the poset $P_{L^\sigma}^{(1,1,1,1,1)}$ 
as the set of all squarefree monomials in $L^\sigma$,
then $x_3x_4x_5$ can not be divided by any of $x_1x_2, x_1x_3, x_1x_4, x_2x_3, x_2x_4$ 
and there are only $9$ degree-$3$ squarefree monomials left for these $5$ degree-$2$ squarefree monomials.
This is an example of a squarefree strongly stable monomial ideal 
with its Stanley depth less than its Hilbert depth.
In Counterexamples \ref{C:cex3} and \ref{C:cex4}, we will see that even if a 
(squarefree) lex ideal is generated by monomials of 
the same degree, its Stanley depth can still be less than its Hilbert depth. 
\end{remark}

Let $I=(x_1^2,\ldots,x_{10}^2)\subset S=K[x_1,\ldots,x_{10}]$, 
then $L=(x_1^2,\ldots,x_1x_{10},x_2^3,\ldots,$ 
$x_2x_{10}^2, \ldots, x_{10}^{11})$
will be the lex ideal in $S$ with the same Hilbert function as $I$ and 
$L^\sigma\subset S'=K[x_1,\ldots, x_{20}]$. There are many generators 
in $L^\sigma$ and there are $20$ variables, so the computation will be very heavy. 
On the other hand, the Hilbert series $H_I(t)$ can be obtained easily
from the resolution of $I$ over $S$. Therefore, it would be handy 
to have an algorithm which computes $\mbox{hdepth}(I)$ directly 
from $H_I(t)$. Next, we will develop such an algorithm (Algorithm \ref{A:algorithm2}) and 
we will use it to compute $\mbox{hdepth}(x_1^2, \ldots, x_{10}^2)$ 
in Example \ref{E:ex4}.

Let $I$ be a proper squarefree monomial ideal in $S=K[x_1,\ldots, x_n]$.
With the notations as in Theorem \ref{T:key}, if $\mbox{hdepth}(I)\geq p$ 
then by part (3) of Theorem \ref{T:key}, 
there exist $b_d, \ldots, b_n\in \mathbb{Z}_{\geq 0}$ such that 
\begin{align*}
&H_I(t)\\
&=\frac{b_dt^d+\cdots+b_pt^p}{(1-t)^p}+\frac{b_{p+1}t^{p+1}}{(1-t)^{p+1}}+\cdots+\frac{b_nt^n}{(1-t)^n}\\
&=\frac{b_dt^d(1-t)^{n-p}+\cdots+b_pt^p(1-t)^{n-p}+b_{p+1}t^{p+1}(1-t)^{n-p-1}+\cdots+b_nt^n}{(1-t)^n}.
\end{align*}
Hence,
\[
H_I(t)=\frac{c_dt^d+\cdots+c_nt^n}{(1-t)^n} \ \mbox{for some } \ c_d, \ldots, c_n \in \mathbb{Z},
\]
and 
\begin{align*}
&c_dt^d+\cdots+c_nt^n\\
&=b_dt^d(1-t)^{n-p}+\cdots+b_pt^p(1-t)^{n-p}+b_{p+1}t^{p+1}(1-t)^{n-p-1}+\cdots+b_nt^n. 
\end{align*}
So we have the following corollary similar to Corollary \ref{C:key}.

\begin{corollary}\label{C:key2}
Let $I$ be a proper squarefree monomial ideal in $S=K[x_1,\ldots, x_n]$.
With the notations as in Theorem \ref{T:key}, we have that
\begin{itemize}
\item[(1)] there exist $c_d, \ldots, c_n \in \mathbb{Z}$ such that 
           \[
           H_I(t)=\frac{c_dt^d+\cdots+c_nt^n}{(1-t)^n};
           \]
\item[(2)] $\mbox{hdepth}(I)$ is equal to the maximal $n-q$ such that 
           \begin{align*}
           &c_dt^d+\cdots+c_nt^n\\
           &=b_dt^d(1-t)^q+\cdots+b_{n-q}t^{n-q}(1-t)^q+b_{n-q+1}t^{n-q+1}(1-t)^{q-1}+\cdots+b_nt^n.
           \end{align*}
           with $b_d, \ldots, b_n\in \mathbb{Z}_{\geq 0}$.
\end{itemize}
\end{corollary}

\begin{remark}\label{R:short}
In part (1) of the above theorem, it is easy to see that $c_d$ is the number 
of monomial generators of $I$ of degree $d$. However, $c_n$ can be $0$. For example,
let $L=(x_1^2, x_1x_2, x_1x_3, x_2^3)$ be a lex ideal in $K[x_1,x_2,x_3]$, then 
$L^\sigma=(x_1x_2, x_1x_3, x_1x_4, x_2x_3x_4)$ is a squarefree lex ideal in $K[x_1, x_2,x_3,x_4]$.
The resolution of $L^\sigma$ is given by the Eliahou-Kervarie resolution \cite{B:EK}, which says:
$x_1x_2$ gives rise to a basis element $(x_1x_2;\emptyset)$ in homological degree $0$;
$x_1x_3$  gives rise to a basis element $(x_1x_3;\emptyset)$ in homological degree $0$ and a basis element $(x_1x_3;2)$ in homological degree $1$;
$x_1x_4$  gives rise to a basis element $(x_1x_4;\emptyset)$ in homological degree $0$, two basis elements $(x_1x_4;2)$ $(x_1x_4;3)$ in homological degree $1$ and a basis element $(x_1x_4;2,3)$ in homological degree $2$; $x_2x_3x_4$  gives rise to a basis element $(x_2x_3x_4;\emptyset)$ in homological degree $0$ and a basis element $(x_2x_3x_4;1)$ in homological degree $1$. So the multigraded Hilbert series of $L^\sigma$ is 
\[
H_{L^\sigma}(T_1,T_2,T_3,T_4)=\frac{T_1T_2+T_1T_3(1-T_2)+T_1T_4(1-T_2)(1-T_3)+T_2T_3T_4(1-T_1)}{(1-T_1)(1-T_2)(1-T_3)(1-T_4)},
\]
and the Hilbert series of $L^\sigma$ is 
\[
H_{L^\sigma}(t)=\frac{t^2+t^2(1-t)+t^2(1-t)^2+t^3(1-t)}{(1-t)^4}=\frac{3t^2-2t^3}{(1-t)^4}.
\]
\end{remark}

It is easy to see that given any $0\leq q \leq n-d$, $b_d, \ldots, b_n$ 
can be uniquely determined by $c_d,\ldots, c_n$ such that the identity in part (2) 
of the above theorem holds. Indeed, we have the following algorithm. 
Algorithm \ref{A:d2} looks long, but later in the examples, we 
will use simple diagrams to represent the computations in this algorithm.

\begin{algorithm}\label{A:d2}
Let $h(t)=c_dt^d+\cdots+c_nt^n$ be a polynomial in $t$ with $c_d,\ldots, c_n\in \mathbb{Z}$. 
Let $0\leq q \leq n-d$.

\noindent Input: $h(t)=c_dt^d+\cdots+c_nt^n$ and $q$.

\noindent Step $1$: Let $b_d=c_d$ and let 
                  \[
                  h_1(t)=h(t)-b_dt^d(1-t)^q.
                  \]
                   Simplify $h_1(t)$ to get 
                 \[
                 h_1(t)=c_{d+1}^{(1)}t^{d+1}+\cdots+c_{n}^{(1)}t^n
                 \]
                 with $c_{d+1}^{(1)},\ldots, c_n^{(1)}\in \mathbb{Z}$.
                  
\noindent Step $2$: Let $b_{d+1}=c_{d+1}^{(1)}$ and let \[ h_2(t)=h_1(t)-b_{d+1}t^{d+1}(1-t)^q.\] Simplify $h_2(t)$ to 
                  get \[h_2(t)=c_{d+2}^{(2)}t^{d+2}+\cdots+c_{n}^{(2)}t^n\] with $c_{d+2}^{(2)},\ldots, c_n^{(2)}\in \mathbb{Z}$.   

\noindent \ldots \ldots \ldots 

\noindent Step $(n-q-d+1)$: Let $b_{n-q}=c_{n-q}^{(n-q-d)}$ and let \[ h_{n-q-d+1}(t)=h_{n-q-d}(t)-b_{n-q}t^{n-q}(1-t)^q. \] Simplify $h_{n-q-d+1}(t)$ to 
                  get \[ h_{n-q-d+1}(t)=c_{n-q+1}^{(n-q-d+1)}t^{n-q+1}+\cdots+c_{n}^{(n-q-d+1)}t^n \] 
                  with $c_{n-q+1}^{(n-q-d+1)},\ldots, c_{n}^{(n-q-d+1)}\in \mathbb{Z}$.    
                  
\noindent Step $(n-q-d+2)$:  Let $b_{n-q+1}=c_{n-q+1}^{(n-q-d+1)}$ and let \[ h_{n-q-d+2}(t)=h_{n-q-d+1}(t)-b_{n-q+1}t^{n-q+1}(1-t)^{q-1}. \] Simplify $h_{n-q-d+2}(t)$ to 
                  get \[ h_{n-q-d+2}(t)=c_{n-q+2}^{(n-q-d+2)}t^{n-q+2}+\cdots+c_{n}^{(n-q-d+2)}t^n \]
                  with $c_{n-q+2}^{(n-q-d+2)},\ldots, c_{n}^{(n-q-d+2)}\in \mathbb{Z}$. 
                  
\noindent Step $(n-q-d+3)$:  Let $b_{n-q+2}=c_{n-q+2}^{(n-q-d+2)}$ and let \[ h_{n-q-d+3}(t)=h_{n-q-d+2}(t)-b_{n-q+2}t^{n-q+2}(1-t)^{q-2}.\] Simplify $h_{n-q-d+3}(t)$ to 
                  get \[ h_{n-q-d+3}(t)=c_{n-q+3}^{(n-q-d+3)}t^{n-q+3}+\cdots+c_{n}^{(n-q-d+3)}t^n \]
                  with $c_{n-q+3}^{(n-q-d+3)},\ldots, c_{n}^{(n-q-d+3)}\in \mathbb{Z}$. 
                  
\noindent \ldots \ldots \ldots 

\noindent Step $(n-d+1)$:  Let $b_{n}=c_{n}^{(n-d)}$.             
                                   
\noindent Output: $b_d,\ldots, b_n\in \mathbb{Z}$.                           
\end{algorithm}

Now let $I, L\in S$ and $L^\sigma \in S'$ be as in Algorithm \ref{A:algorithm1}.
By Remark \ref{R:relation} we see that if 
\[
H_{L^\sigma}(t)=\frac{c_dt^d+\cdots+c_mt^m}{(1-t)^m},
\]
then
\[
H_I(t)=\frac{c_dt^d+\cdots+c_mt^m}{(1-t)^n}.
\]
Hence, for a given $0\leq q\leq m-d$, if by Algorithm \ref{A:d2} 
we get $b_d, \ldots, b_m$ such that $b_i<0$ for some $d\leq i\leq m$, 
then by part (3) of Theorem \ref{T:key} we have 
$\mbox{hdepth}(L^\sigma)<m-q$, so that
\[
\mbox{hdepth}(I)=\mbox{hdepth}(L^\sigma)-(m-n)<n-q.
\]
On the other hand, if there exists $b_d, \ldots, b_m \in \mathbb{Z}_{\geq 0}$ and $\beta_d, \ldots, \beta_m\in \mathbb{Z}_{\geq 0}$
such that
\[
c_dt^d+\cdots+c_mt^m=\sum_{i=d}^mb_it^i(1-t)^{\beta_i},
\]
then $\mbox{hdepth}(I)\geq n-q$ where $q=\max \{\beta_i \mid d\leq i\leq m\}$.
By these observations, we have the following algorithm for
computing the Hilbert depth of a graded ideal. 

\begin{algorithm}\label{A:algorithm2}
Let $I$ be a graded ideal in the polynomial ring $S=K[x_1,\ldots,x_n]$.

\noindent Input: $I$

\noindent Step 1: Let $L$ be the lex ideal in $S$ with the same Hilbert function as $I$.
                  Find the minimal generators of $L$ and denote them by $u_1,\ldots, u_s$.
                  Let $m=\max\{\mbox{m}(u_i)+\deg(u_i)-1\mid 1\leq i \leq s \}$.
                  
\noindent Step 2: Find the Hilbert series of $I$:
\[
H_I(t)=\frac{c_dt^d+\cdots+c_rt^r}{(1-t)^n}. 
\]
(Note that $r\leq m$ and in Remark \ref{R:short} we have an example with $r<m$.)
Set $Q(t):=c_dt^d+\cdots+c_mt^m$, where $c_{r+1}=\cdots=c_m=0$ if $r<m$. Set $q:=0$.

\noindent Step 3: Apply algorithm \ref{A:d2} to $Q(t)$ and $q$, and we get $b_d, \ldots, b_m$. 
If $b_d, \ldots, b_{m}\in \mathbb{Z}_{\geq 0}$ then put 
$\mbox{hdepth}(I)=n-q$ and stop; otherwise, set $q:=q+1$ and run step 3 again.
                  
\noindent Output: $\mbox{hdepth}(I)$.                                  
\end{algorithm}

In the following examples, we will compute the Hilbert depth of some graded ideals by using Algorithm \ref{A:algorithm2}, 
and we will compare  Algorithm \ref{A:algorithm2} with the algorithms in \cite{B:Po} and \cite{B:BMU}.

\begin{example}\label{E:ex2}
Let $I=(x_1^2, x_1x_2, x_1x_3, x_2^2, x_3^2)\subset S=K[x_1,x_2,x_3]$ 
be as in Example \ref{E:ex1}, where we have computed 
$\mbox{hdepth}(I)=2$ by Algorithm \ref{A:algorithm1}. 
Now we will use Algorithm \ref{A:algorithm2} to compute $\mbox{hdepth}(I)$. 
Let $L$ be the lex ideal in $S$ with the same Hilbert function as $I$, then
\[
L=(x_1^2, x_1x_2, x_1x_3, x_2^2, x_2x_3, x_3^3) \ \mbox{and}\ m=5.
\]
The resolution of $L$ is given by the Eliahou-Kervarie resolution \cite{B:EK}, which says:
$x_1^2$  gives rise to a basis element $(x_1^2;\emptyset)$ in homological degree $0$;
$x_1x_2$  gives rise to a basis element $(x_1x_2;\emptyset)$ in homological degree $0$ and a basis element $(x_1x_2;1)$ in homological degree $1$;
$x_1x_3$  gives rise to a basis element $(x_1x_3;\emptyset)$ in homological degree $0$, two basis elements $(x_1x_3;1)$ $(x_1x_3;2)$ in homological degree $1$ and a basis element $(x_1x_3;1,2)$ in homological degree $2$; 
$x_2^2$  gives rise to a basis element $(x_2^2;\emptyset)$ in homological degree $0$ and a basis element $(x_2^2;1)$ in homological degree $1$;
$x_2x_3$  gives rise to a basis element $(x_2x_3;\emptyset)$ in homological degree $0$, two basis elements $(x_2x_3;1)$ $(x_2x_3;2)$ in homological degree $1$ and a basis element $(x_2x_3;1,2)$ in homological degree $2$; 
$x_3^3$  gives rise to a basis element $(x_3^3;\emptyset)$ in homological degree $0$, two basis elements $(x_3^3;1)$ $(x_3^3;2)$ in homological degree $1$ and a basis element $(x_3^3;1,2)$ in homological degree $2$. So the multigraded Hilbert series of $L$ is 
\begin{align*}
&H_{L}(T_1,T_2,T_3)\\
&=\begin{matrix}\frac{T_1^2+T_1T_2(1-T_1)+T_1T_3(1-T_1)(1-T_2)+T_2^2(1-T_1)+T_2T_3(1-T_1)(1-T_2)+T_3^3(1-T_1)(1-T_2)}{(1-T_1)(1-T_2)(1-T_3)}\end{matrix},
\end{align*}
and the Hilbert series of $I$ is 
\begin{align*}
H_I(t)=H_L(t)&=\frac{t^2+t^2(1-t)+t^2(1-t)^2+t^2(1-t)+t^2(1-t)^2+t^3(1-t)^2}{(1-t)^3}\\
             &=\frac{5t^2-5t^3+t^5}{(1-t)^3}.
\end{align*}
Run Algorithm \ref{A:algorithm2}. $q=0$ does not work because $b_3=-5<0$. 
Then we try $q=1$, and the computation 
of $b_2, b_3, b_4, b_5$ by Algorithm \ref{A:d2} is shown in the following diagram::
\[
\begin{matrix}
\hline
5&-5&0&1\\
5&-5&\ &\ \\
\hline
\ &\ &\ &1
\end{matrix}
\]
we see that $b_2=5, b_3=b_4=0, b_5=1$. 
So $q=1$ works and $\mbox{hdepth}(I)=3-1=2$ as expected. 

By the algorithm in \cite{B:Po}, we compute
\[
\frac{5t^2-5t^3+t^5}{1-t}=5t^2+0t^3+0t^4+t^5+\cdots \tag{1}
\]
and then $\mbox{hdepth}(I)=3-1=2$.

To use the algorithm in \cite{B:BMU}, we set $Q(t)=5t^2-5t^3+t^5$. 
Then $\tilde{Q}(t)=1-5t^3+5t^4-t^5$,
 $\delta_3(\tilde{Q})=0$ and $e=\max\{0,6\}=6$. 
Therefore, by the computation in (1) we have $\mbox{hdepth}(I)=3-1=2$.
\end{example}

\begin{example}\label{E:ex3}
Let $\mathfrak{m}=(x_1,\ldots, x_5)\subset K[x_1, \ldots, x_5]$. By \cite{B:BKU2}, we know
$\mbox{hdepth}(\mathfrak{m})=\lceil \frac{5}{2} \rceil=3$. 
Now we use Algorithm \ref{A:algorithm2} to compute $\mbox{hdepth}(\mathfrak{m})$. 
First, we have 
\[
H_{\mathfrak{m}}(t)=\frac{1}{(1-t)^5}-1=\frac{5t-10t^2+10t^3-5t^4+t^5}{(1-t)^5} \ \mbox{and}\ m=5.
\]
Obviously, $q=0$ does not work. If $q=1$, then by the following diagram
\[
\begin{matrix}
\hline
5&-10&10&-5&1\\
5&-5&\ &\ &\ \\
\hline
\ &-5&10&-5&1
\end{matrix}
\]
we see that $b_2=-5<0$. Hence, $q=1$ does not work. 
Next we try $q=2$ and Algorithm \ref{A:d2} is represented by the following diagram:
\[
\begin{matrix}
\hline
5&-10&10&-5&1\\
5&-10&5&\ &\ \\
\hline
\ &\ &5&-5&1\\
\ &\ &5&-10&5\\
\hline
\ &\ &\ &5&-4\\
\ &\ &\ &5&-5\\
\hline
\ &\ &\ &\ &1
\end{matrix}
\]
Hence, $b_1=5, b_2=0, b_3=5, b_4=5, b_5=1$, so that $q=2$ works and 
$\mbox{hdepth}(\mathfrak{m})=5-2=3$ as expected. 

Note that there is an easier way to show that $q=2$ works, 
that is, we have the following diagram:
\[
\begin{matrix}
\hline
5&-10&10&-5&1\\
5&-10&5&\ &\ \\
\hline
\ &\ &5&-5&1\\
\ &\ &5&-5&\ \\
\hline
\ &\ &\ &\ &1
\end{matrix}
\]
Hence, we have $5t-10t^2+10t^3-5t^4+t^5=5t(1-t)^2+5t^3(1-t)+t^5$, 
which also implies that $q=2$ works.

By the algorithm in \cite{B:Po}, we compute
\begin{align*}
\frac{5t-10t^2+10t^3-5t^4+t^5}{1-t}&=5t-5t^2+5t^3+0t^4+t^5+\cdots \\
\frac{5t-10t^2+10t^3-5t^4+t^5}{(1-t)^2}&=5t+0t^2+5t^3+5t^4+6t^5+\cdots \tag{2}
\end{align*}
and then $\mbox{hdepth}(I)=5-2=3$.

To use the algorithm in \cite{B:BMU}, we set $Q(t)=5t-10t^2+10t^3-5t^4+t^5$. 
Then $\tilde{Q}(t)=1-t^5$,
 $\delta_5(\tilde{Q})=0$ and $e=\max\{0,6\}=6$. 
Therefore, by the computation in (2) we have $\mbox{hdepth}(I)=5-2=3$.
\end{example}

Algorithm \ref{A:algorithm2} depends not only on the Hilbert series $H_I(t)$ 
but also on the number $m$. However, after computing many examples the author 
was unable to find an example where the number $m$ really matters. 
Just like the above two examples, most of the time $m$ and $r$ are the same, 
so  the calculation depends only on $H_I(t)$. 
In the rare cases when $r< m$, the following two examples suggest that we could still
ignore $m$.

\begin{example}\label{E:r1}
Let $L=(x_1^2, x_1x_2, x_1x_3, x_2^3)\subset K[x_1, x_2, x_3]$ be as in Remark \ref{R:short},
then we have 
\[
H_L(t)=\frac{3t^2-2t^3}{(1-t)^3}, \ \mbox{and} \ r=3<m=4.
\]
Run Algorithm \ref{A:algorithm2}. $q=0$ does not work because $b_3=-2<0$. Then we try $q=1$. The computation 
of $b_2, b_3, b_4$ by Algorithm \ref{A:d2} is shown in the following diagram:
\[
\begin{matrix}
\hline
3&-2&0\\
3&-3&\ \\
\hline
\ &1&0\\
\ &1&-1\\
\hline
\ &\ & 1
\end{matrix}
\]
Then $b_2=3,b_3=1,b_4=1$, so that $q=1$ works and $\mbox{hdepth}(L)=3-1=2$. 
However, even if we do not know $m=4$, 
we can still see that $q=0$ does not work and for $q=1$, 
from the following diagram,
\[
\begin{matrix}
\hline
3&-2\\
3&-3\\
\hline
\ &1
\end{matrix}
\]
we see that $3t^2-2t^3=3t^2(1-t)+t^3$, 
which implies $q=1$ works. So, $\mbox{hdepth}(L)=3-1=2$.

By the algorithm in \cite{B:Po}, we compute
\[
\frac{3t^2-2t^3}{1-t}=3t^2+t^3+\cdots \tag{3}
\]
and then $\mbox{hdepth}(I)=3-1=2$.

To use the algorithm in \cite{B:BMU}, we set $Q(t)=3t^2-2t^3$. 
Then $\tilde{Q}(t)=1-3t^2+2t^3$,
 $\delta_3(\tilde{Q})=2$ and $e=\max\{2,4\}=4$. 
Therefore, by the computation in (3) we have $\mbox{hdepth}(I)=3-1=2$.
\end{example}

\begin{example}\label{E:r2}
In the previous example $r=m-1$. The simplest example one can find with $r=m-2$ is the following.
Let $L$ be the lex ideal in $K[x_1, x_2, x_3, x_4]$ minimally generated by $x_1^2, x_1x_2, x_1x_3, x_1x_4, $ $ x_2^2, x_2x_3, x_2x_4^2,$ $ x_3^4$,
then we have 
\[
H_L(t)=\frac{6t^2-8t^3+3t^4}{(1-t)^4}, \ \mbox{and} \ r=4<m=6.
\]
Run Algorithm \ref{A:algorithm2}. $q=0$ does not work because $b_3=-8<0$.  $q=1$ does not work because $b_3=-2<0$.Then we try $q=2$. 
By the following diagram:
\[
\begin{matrix}
\hline
6&-8&3 \\
6&-12&6 \\
\hline
\ &4&-3 \\
\ &4&-4 \\
\hline
\ &\  &1 
\end{matrix}
\]
we see that $6t^2-8t^3+3t^4=6t^2(1-t)^2+4t^3(1-t)+t^4$. 
Hence, $q=2$ works and $\mbox{hdepth}(L)=4-2=2$. 
Note that the computation works even if we do not know $m=6$.
And we have a decomposition of $H_L(t)$:
\[
H_L(t)=\frac{6t^2}{(1-t)^2}+\frac{4t^3}{(1-t)^3}+\frac{t^4}{(1-t)^4}.
\]

By the algorithm in \cite{B:Po}, we compute
\begin{align*}
\frac{6t^2-8t^3+3t^4}{1-t}&=6t^2-2t^3+t^4+\cdots \\
\frac{6t^2-8t^3+3t^4}{(1-t)^2}&=6t^2+4t^3+5t^4+\cdots \tag{4}
\end{align*}
and then $\mbox{hdepth}(L)=4-2=2$.

To use the algorithm in \cite{B:BMU}, we set $Q(t)=6t^2-8t^3+3t^4$. 
Then $\tilde{Q}(t)=1-4t^3+3t^4$, and we can calculate that
 $\delta_4(\tilde{Q})=4$ and $e=\max\{4,5\}=5$. 
Therefore, by the computation in (4) we have $\mbox{hdepth}(L)=4-2=2$.
And by the method of \cite{B:BMU}, we get a decomposition of $H_L(t)$:
\[
H_L(t)=\frac{6t^2+4t^3+5t^4}{(1-t)^2}+\frac{5t^5}{(1-t)^3}+\frac{t^5}{(1-t)^4},
\]
which is different from the decomposition previously obtained by Algorithm \ref{A:algorithm2}. 
\end{example}

\begin{remark}\label{R:comparison}
From the above four examples, we can see that Algorithm \ref{A:algorithm2}
is different from the algorithms in \cite{B:Po} and \cite{B:BMU}. 
Also, the author feels that in general, $m$ is not needed when applying  
Algorithm \ref{A:algorithm2} to graded ideals. For general modules over 
the polynomial ring, it is a different story and the following example is interesting. 

Let $M=K\bigoplus x_1^3K[x_1,x_2,x_3]$ be a module over $K[x_1,x_2,x_3]$,
then 
\[
H_M(t)=\frac{1-3t+3t^2}{(1-t)^3}.
\]
If we  calculate $\mbox{hdepth}(M)$ by a method similar to Algorithm \ref{A:algorithm2},
we have the following diagram:
\[
\begin{matrix}
\hline
1&-3&3&0\\
1&-3&3&-1  \\
\hline
\ &\ &\ &1 
\end{matrix}
\]
which imples that $1-3t+3t^2=(1-t)^3+t^3$ and then $\mbox{hdepth}(M)=3-3=0$. 
Note that in the above diagram, $0$ must be added to the end of the first row;
otherwise, one can not proceed the calculation. 
However, it is easy to see that $H_M(t)$ can not be 
the Hilbert series of a graded ideal in a polynomial ring. 

To use the algorithm in \cite{B:Po}, we compute
\begin{align*}
\frac{1-3t+3t^2}{1-t}&=1-2t+t^2+t^3+\cdots \\
\frac{1-3t+3t^2}{(1-t)^2}&=1-t+0t^2+t^3+\cdots \\
\frac{1-3t+3t^2}{(1-t)^3}&=1+0t+0t^2+t^3+\cdots
\end{align*}
and then $\mbox{hdepth}(M)=3-3=0$.

To use the algorithm in \cite{B:BMU}, we set $Q(t)=1-3t+3t^2$. 
Then $\tilde{Q}(t)=1-3t+3t^2$. We can calculate that
 $\delta_3(\tilde{Q})=3$ and $e=\max\{3,3\}=3$. 
Hence, by the following computation: 
\begin{align*}
\frac{1-3t+3t^2}{1-t}&=1-2t+t^2+\cdots \\
\frac{1-3t+3t^2}{(1-t)^2}&=1-t+0t^2+\cdots \\
\frac{1-3t+3t^2}{(1-t)^3}&=1+0t+0t^2+\cdots
\end{align*}
we have $\mbox{hdepth}(M)=3-3=0$.

When applying the algorithm in \cite{B:Po}, one needs to decide 
the number of terms to be calculated in each formal power series, 
which may increase during the computation;
when applying the algorithm in \cite{B:BMU}, 
the number of terms to be calculated in each formal power series is 
given by the fixed number $e=\max\{\delta_d(\tilde{Q}),\deg(Q)+1\}$. 

By the previous examples and some other examples we have computed, 
we wonder if $\delta_d(\tilde{Q})\leq \deg(Q)+1$ 
holds for all graded ideals in the polynomial ring. This is a question in some sense 
similar to the one about the necessity of $m$ in Algorithm \ref{A:algorithm2}.
\end{remark}

\begin{example}\label{E:ex4}
Let $I=(x_1^2,\ldots,x_{10}^2)\subset S=K[x_1,\ldots,x_{10}]$. 
As explained after Remark \ref{R:remarkex1}, it is not easy to compute
$\mbox{hdepth}(I)$ by Algorithm \ref{A:algorithm1}.
However, since $I$ is a complete intersection monomial ideal, it is easy to see that 
\begin{align*}
&H_I(t)\\
&=\begin{matrix}\frac{\binom{10}{1}t^2-\binom{10}{2}t^4+\binom{10}{3}t^6-\binom{10}{4}t^8+\binom{10}{5}t^{10}-\binom{10}{6}t^{12}+\binom{10}{7}t^{14}-\binom{10}{8}t^{16}+\binom{10}{9}t^{18}-\binom{10}{10}t^{20}}{(1-t)^{10}}\end{matrix} \\
&=\begin{matrix}\frac{10t^2-45t^4+120t^6-210t^8+252t^{10}-210t^{12}+120t^{14}-45t^{16}+10t^{18}-t^{20}}{(1-t)^{10}}\end{matrix}.
\end{align*}

By \cite{B:Ci} or \cite{B:Sh1}, we know that the Stanley depth of $I$ is 
\[
\mbox{sdepth}(x_1^2,\ldots,x_{10}^2)=\mbox{sdepth}(x_1,\ldots,x_{10})=\lceil \frac{10}{2} \rceil=5.
\]
One may wonder if $\mbox{hdepth}(I)=\mbox{sdepth}(I)=5$. 
Next we will use Algorithm \ref{A:algorithm2} to compute $\mbox{hdepth}(I)$.

First, $q=0$ does not work because $b_4=-45<0$. 
If $q=1$, then by the following diagram,
\[
\begin{matrix}
\hline
10&0&-45&\cdots\\
10&-10&\ &\ \\
\hline
\ &10&-45&\cdots\\
\ &10&-10&\ \\
\hline
\ &\ &-35&\cdots 
\end{matrix}
\]
we have $b_4=-35<0$ which does not work. 
If $q=2$, then by the following diagram,
\[
\begin{matrix}
\hline
10&0&-45&0&\cdots\\
10&-20&10&\ &\ \\
\hline
\ &20&-35&0&\cdots\\
\ &20&-40&20&\  \\
\hline
\ &\ &5&-20&\cdots \\
\ &\ &5&-10&5 \\
\hline
\ &\ &\ &-10&\cdots
\end{matrix}
\]
we have $b_5=-10<0$ which does not work. 
If $q=3$, then by the following diagram,
\[
\begin{matrix}
\hline
10&0&-45&0&\cdots\\
10&-30&30&-10 &\ \\
\hline
\ &30&-75&10&\cdots\\
\ &30&-90&90&30  \\
\hline
\ &\ &15&-80&\cdots \\
\ &\ &15&-45&45\\
\hline
\ &\ &\ &-35&\cdots
\end{matrix}
\]
we have $b_5=-35<0$ which does not work.
If $q=4$ then we have that
\begin{align*}
&10t^2-45t^4+120t^6-210t^8+252t^{10}-210t^{12}+120t^{14}-45t^{16}+10t^{18}-t^{20}\\
&=(10t^2+40t^3+55t^4+115t^8+100t^9)(1-t)^4\\
&\ +(20t^5+120t^7+82t^{10}+106t^{11}+147t^{12}+105t^{13}+100t^{14})(1-t)^3\\
&\ +(132t^{15}+24t^{16})(1-t)^2+(16t^{17}+2t^{18}+2t^{19})(1-t)+t^{20}.
\end{align*} 
Hence, $q=4$ works. So we have that
\[
\mbox{hdepth}(x_1^2,\ldots,x_{10}^2)=10-4=6>\mbox{sdepth}(x_1^2,\ldots,x_{10}^2)=5.
\]
Note that from either algebraic or combinatorial point of view, 
$I$ is a nice and simple monomial ideal, but even in this case 
its Stanley depth and Hilbert depth are not equal to each other. 
\end{example}

From the above examples, we see that Algorithm \ref{A:algorithm1}
and Algorithm \ref{A:algorithm2} are useful tools for computing
the Hilbert depth of a graded ideal, especially when the ideal is a 
(squarefree) lex ideal, or when the Hilbert series of the ideal can 
be easily obtained. More importantly, without  Algorithm \ref{A:algorithm2}, 
it would be impossible for the author to find the counterexamples 
in the next section.

\section{Some Counterexamples}
In \cite{B:Sh2}, Shen made some conjectures related to the 
Hilbert depth of a (squarefree) lex ideal which is generated by monomials of the same degree.
In this section we will give some counterexamples to these conjectures, 
and Algorithm \ref{A:algorithm2} will be used in some of the computations. 

Conjecture 3.5 in \cite{B:Sh2} says that if $I$ is a stable ideal in $S=K[x_1,\ldots,x_n]$ 
generated by monomials of degree $d$ and $H_I(t)=a_dt^d+a_{d+1}t^{d+1}+\cdots$, 
then $\mbox{hdepth}(I)=\lfloor \frac{a_{d+1}}{a_d} \rfloor$. 

By Theorem \ref{T:Ul} part (1), we can see easily that for a graded ideal $I$,
if $H_I(t)=a_dt^d+a_{d+1}t^{d+1}+\cdots$ with $a_d\neq 0$, 
then $\mbox{hdepth}(I)\leq \lfloor \frac{a_{d+1}}{a_d} \rfloor$. 
The equal sign is obtained in the case of powers 
of maximal ideals $\mathfrak{m}^d$ (by \cite{B:BKU2}) 
and the case of squarefree Veronese ideals $I_{n,d}$ (by \cite{B:GLW}). 
But in general, the equal sign can not be obtained. 
For example, let $I=(x_1^2, \ldots, x_{10}^2)$ be as in 
Example \ref{E:ex4}, then $H_I(t)=10t^2+100t^3+\cdots$, 
but $\mbox{hdepth}(I)=6\neq 10$. How about stable ideals?
In the next counterexample, we will see that the identity 
may not hold even for a (squarefree) lex ideal generated by monomials of the same degree. 

\begin{counterexample}\label{C:cex1}
Let $L=(x_1^2, x_1x_2, \ldots, x_1x_{100}, x_2^2, x_2x_3, \ldots, x_2x_{13})\subset S=K[x_1, x_2, \ldots, x_{100}]$. 
Then $L$ is a lex ideal in $S$ generated by some monomials of degree $2$.
Similar to the analysis in Example \ref{E:ex2}, by the 
Eiliahou-Kervaire resolution of $L$, we have that
\begin{align*}
&H_L(t)\\
&=\frac{t^2\left( 1+(1-t)+\cdots+(1-t)^{99}+(1-t)+(1-t)^2+\cdots+(1-t)^{12}\right)}{(1-t)^{100}}\\
&=\frac{t^2}{(1-t)^{100}}\left(\frac{1-(1-t)^{100}}{1-(1-t)}+\frac{1-(1-t)^{13}}{1-(1-t)}-1\right)\\
&=\frac{t^2}{(1-t)^{100}}\left(112-5028t+161986t^2-3921940t^3+\cdots-t^{99} \right)\\
&=\frac{112t^2-5028t^3+161986t^4-3921940t^5+\cdots -t^{101}}{(1-t)^{100}}\\
&=112t^2+6172t^3+\cdots.
\end{align*}
Assume that $\mbox{hdepth}(L)=\lfloor \frac{6172}{112} \rfloor =55$, 
then by the analysis before Corollary \ref{C:key2} we conclude that 
$q=100-55=45$ would work in step 3 of Algorithm \ref{A:algorithm2}. 
However, the following diagram
\[
\begin{matrix}
\hline
112&-5028&161986&-3921940&\cdots\\
112&-112\times 45 &112\times \binom{45}{2} &-112\times \binom{45}{3} &\cdots \\
\hline
\ &12&51106&-2332660&\cdots\\
\ &12&-12\times 45&12\times \binom{45}{2}&\cdots  \\
\hline
\ &\ &51646&-2344540&\cdots \\
\ &\ &51646&-51646\times 45&\cdots\\
\hline
\ &\ &\ &-20470&\cdots
\end{matrix}
\]
tells us that $b_5=-20470<0$, so that $q=45$ does not work, which is a contradiction. 
Therefore, $\mbox{hdepth}(L)\neq \lfloor \frac{6172}{112} \rfloor $ and $\mbox{hdepth}(L)\leq 54$.

Note that $L^\sigma=(x_1x_2, x_1x_3, \ldots, x_1x_{101}, x_2x_3, x_2x_4, \ldots, x_2x_{14})$
is a squarefree lex ideal in $S'=K[x_1, \ldots, x_{101}]$ 
generated by some squarefree monomials of degree $2$, and 
\begin{align*}
H_{L^\sigma}(t)&=H_L(t)\frac{1}{1-t}\\
               &=(112t^2+6172t^3+\cdots)(1+t+\cdots)\\
               &=112t^2+6284t^3+\cdots.
\end{align*}
So, $\mbox{hdepth}(L^\sigma)=\mbox{hdepth}(L)+1\leq 54+1=55$ 
and then $\mbox{hdepth}(L^\sigma)\neq \lfloor \frac{6284}{112} \rfloor =56$.
\end{counterexample}

Conjecture 5.5 in \cite{B:Sh2} says that if $L_1$ and $L_2$ are lex ideals in $S$ 
both generated by monomials of degree $d$ and $L_1\subset L_2$, 
then $\mbox{hdepth}(L_1)\geq \mbox{hdepth}(L_2)$. 
The following is a counterexample to this conjecture.

\begin{counterexample}\label{C:cex2}
In the polynomial ring $S=K[x_1,x_2,\ldots, x_{10}]$, let
\begin{align*}
L_1&=(x_1^2, x_1x_2, x_1x_3, x_1x_4,x_1x_5, x_1x_6, x_1x_7,x_1x_8, x_1x_9, x_1x_{10}, x_2^2)\\
L_2&=(x_1^2, x_1x_2, x_1x_3, x_1x_4,x_1x_5, x_1x_6, x_1x_7,x_1x_8, x_1x_9, x_1x_{10}, x_2^2, x_2x_3)
\end{align*}
be two lex ideals generated by monomials of degree $2$ such that $L_1\subset L_2$.
Similar to the analysis in Example \ref{E:ex2}, by the 
Eiliahou-Kervaire resolution, we have that
\begin{align*}
&H_{L_1}(t)\\
&=\frac{t^2\left( 1+(1-t)+\cdots+(1-t)^{9}+(1-t)\right)}{(1-t)^{10}}\\
&=\frac{t^2}{(1-t)^{10}}\left(\frac{1-(1-t)^{10}}{1-(1-t)}+(1-t)\right)\\
&=\frac{11t^2-46t^3+120t^4-210t^5+252t^6-210t^7+120t^8-45t^9+10t^{10}-t^{11} }{(1-t)^{10}},\\
&H_{L_2}(t)\\
&=\frac{t^2\left( 1+(1-t)+\cdots+(1-t)^{9}+(1-t)+(1-t)^2\right)}{(1-t)^{10}}\\
&=\frac{t^2}{(1-t)^{10}}\left(\frac{1-(1-t)^{10}}{1-(1-t)}+(1-t)+(1-t)^2\right)\\
&=\frac{12t^2-48t^3+121t^4-210t^5+252t^6-210t^7+120t^8-45t^9+10t^{10}-t^{11} }{(1-t)^{10}}.
\end{align*}

Apply Algorithm \ref{A:algorithm2} to $L_1$. It is easy to see 
that if $q\leq 4$ then $b_3<0$, so that $q\leq 4$ does not work. 
On the other hand, we have that
\begin{align*}
&11t^2-46t^3+120t^4-210t^5+252t^6-210t^7+120t^8-45t^9+10t^{10}-t^{11}\\
&=(11t^2+9t^3)(1-t)^5+(55t^4+30t^5)(1-t)^4+77t^6(1-t)^3\\
&\ +(27t^7+17t^8)(1-t)^2+(9t^9+2t^{10})(1-t)+t^{11},
\end{align*}
which implies that $q=5$ works. Therefore, $\mbox{hdepth}(L_1)=10-5=5$.

Apply Algorithm \ref{A:algorithm2} to $L_2$. It is easy to see 
that if $q\leq 3$ then $b_3<0$, so that $q\leq 3$ does not work. 
On the other hand, we have that
\begin{align*}
&12t^2-48t^3+121t^4-210t^5+252t^6-210t^7+120t^8-45t^9+10t^{10}-t^{11}\\
&=(12t^2+49t^4+34t^5)(1-t)^4+82t^6(1-t)^3\\
&\ +(28t^7+17t^8)(1-t)^2+(9t^9+2t^{10})(1-t)+t^{11},
\end{align*}
which implies that $q=4$ works. Therefore, $\mbox{hdepth}(L_2)=10-4=6$.

So we have found lex ideals $L_1$ and $L_2$ of degree $2$ 
such that $L_1\subset L_2$ but $\mbox{hdepth}(L_1)<\mbox{hdepth}(L_2)$.

Note that $L_1^\sigma$ and $L_2^\sigma$ are squarefree lex ideals 
in $S'=K[x_1,\ldots, x_{11}]$. 
Both are generated by some monomials of degree $2$, and $L_1^\sigma\subset L_2^\sigma$. 
Since $\mbox{hdepth}(L_1^\sigma)=\mbox{hdepth}(L_1)+1=6$ and $\mbox{hdepth}(L_2^\sigma)=\mbox{hdepth}(L_2)+1=7$,
it follows that $\mbox{hdepth}(L_1^\sigma)<\mbox{hdepth}(L_2^\sigma)$.
\end{counterexample}

Conjecture 4.3 in \cite{B:Sh2} says that 
if $I$ is a strongly stable monomial ideal in $S=K[x_1,\ldots, x_n]$ generated by monomials of the same degree 
and $I^\sigma$ is a squarefree strongly stable monomial ideal in $S'=K[x_1,\ldots, x_m]$,
then we have that 
\begin{align*}
\mbox{sdepth}(I)&=\mbox{hdepth}(I), \\
\mbox{sdepth}(I^\sigma)&=\mbox{hdepth}(I^\sigma),\\
\mbox{sdepth}(I)&=\mbox{sdepth}(I^\sigma)-(m-n). 
\end{align*}
In the following two counterexamples, we will see that the first two identities do not hold.

\begin{counterexample}\label{C:cex3}
Let $L_2=(x_1^2, x_1x_2, \ldots, x_1x_{10}, x_2^2, x_2x_3)$ 
be the lex ideal in $S=K[x_1, \ldots, x_{10}]$. 
From Counterexample \ref{C:cex2},We already know that $\mbox{hdepth}(L_2)=6$.

To compute $\mbox{sdepth}(L_2)$, we use the method of \cite{B:HVZ} and let $g=(2,2,1,\ldots,1)$. 
Instead of thinking of $P_{L_2}^g$ as a subset of $\mathbb{N}^{10}$, 
we will equivalently view $P_{L_2}^g$ as the set of all monomials 
which divides $x_1^2x_2^2x_3\cdots x_{10}$ and can be divided by one of the generators of $L_2$. 

All the degree $2$ monomials in $P_{L_2}^g$ are
\[
x_1^2, x_1x_2, x_1x_3, x_1x_4, x_1x_5, x_1x_6, x_1x_7, x_1x_8, x_1x_9, x_1x_{10}, x_2^2, x_2x_3.
\]
All the degree $3$ monomials in $P_{L_2}^g$ are
\[
\begin{matrix}
x_1^2x_2&x_1^2x_3&x_1^2x_4&x_1^2x_5&x_1^2x_6&x_1^2x_7&x_1^2x_8&x_1^2x_9\\
x_1^2x_{10}&x_1x_2^2&x_1x_2x_3&x_1x_2x_4&x_1x_2x_5&x_1x_2x_6&x_1x_2x_7&x_1x_2x_8\\
x_1x_2x_9&x_1x_2x_{10}&x_1x_3x_4&x_1x_3x_5&x_1x_3x_6&x_1x_3x_7&x_1x_3x_8&x_1x_3x_9\\
x_1x_3x_{10}&x_1x_4x_5&x_1x_4x_6&x_1x_4x_7&x_1x_4x_8&x_1x_4x_9&x_1x_4x_{10}&x_1x_5x_6\\
x_1x_5x_7&x_1x_5x_8&x_1x_5x_9&x_1x_5x_{10}&x_1x_6x_7&x_1x_6x_8&x_1x_6x_9&x_1x_6x_{10}\\
x_1x_7x_8&x_1x_7x_9&x_1x_7x_{10}&x_1x_8x_9&x_1x_8x_{10}&x_1x_9x_{10}&\ &\ &\ \\
x_2^2x_3&x_2^2x_4&x_2^2x_5&x_2^2x_6&x_2^2x_7&x_2^2x_8&x_2^2x_9&x_2^2x_{10}\\
x_2x_3x_4&x_2x_3x_5&x_2x_3x_6&x_2x_3x_7&x_2x_3x_8&x_2x_3x_9&x_2x_3x_{10}&\ .
\end{matrix}
\]
There are $12$ degree-$2$ monomials and $61$ degree-$3$ monomials in $P_{L_2}^g$.

Assume that $\mbox{sdepth}(L_2)=6$. Then there exists a partition $\mathcal{P}$
of $P_{L_2}^g$ such that the Stanley decomposition induced by $\mathcal{P}$ 
has Stanley depth $6$. Suppose that in the partition $\mathcal{P}$ we have 
intervals 
$[x_1^2, x_1^2m_1]$, $[x_1x_2, x_1x_2m_2]$, $[x_1x_3, x_1x_3m_3]$, 
$\ldots$, $[x_1x_{10}, x_1x_{10}m_{10}]$. Next we will show that each of these $10$
intervals has at least $5$ degree-$3$ monomials. 
\begin{itemize}
\item[(1)] $[x_1^2, x_1^2m_1]$: since $\mbox{sdepth}(L_2)=6$, 
           there exist $2\leq i_1< \cdots < i_5$ and $m_1'$
           such that $m_1=x_{i_1}\cdots x_{i_5}m_1'$. 
           Thus, $x_1^2x_{i_1}, \ldots, x_1^2x_{i_5}$ are $5$ degree-$3$ 
           monomials in this interval. 
\item[(2)] $[x_1x_2, x_1x_2m_2]$: if $x_1, x_2$ both can not divide $m_2$, then 
           there exist $3\leq i_1 < \cdots < i_6$ and $m_2'$
           such that $m_2=x_{i_1}\cdots x_{i_6}m_2'$, so that 
           $x_1x_2x_{i_1}, \ldots, x_1x_2x_{i_6}$ are 
           $6$ degree-$3$ monomials in this interval; 
           if $x_1$ divides $m_2$ and $x_2$ can not divide $m_2$, then 
           there exist $3\leq i_1 < \cdots < i_5$ and $m_2'$
           such that $m_2=x_1x_{i_1}\cdots x_{i_5}m_2'$, so that 
           $x_1^2x_2,x_1x_2x_{i_1}, \ldots, x_1x_2x_{i_5}$ are 
           $6$ degree-$3$ monomials in this interval;  
           if $x_1$ can not divide $m_2$ and $x_2$ divides $m_2$, then 
           there exist $3\leq i_1 < \cdots < i_5$ and $m_2'$
           such that $m_2=x_2x_{i_1}\cdots x_{i_5}m_2'$, so that 
           $x_1x_2^2,x_1x_2x_{i_1}, \ldots, x_1x_2x_{i_5}$ are 
           $6$ degree-$3$ monomials in this interval;  
           if $x_1, x_2$ both divides $m_2$, then 
           there exist $3\leq i_1 < \cdots < i_4$ and $m_2'$
           such that $m_2=x_1x_2x_{i_1}\cdots x_{i_4}m_2'$, so that 
           $x_1^2x_2,x_1x_2^2,x_1x_2x_{i_1}, \ldots, x_1x_2x_{i_4}$ are 
           $6$ degree-$3$ monomials in this interval;
\item[(3)] $[x_1x_j, x_1x_jm_j]$ with $3\leq j \leq 10$: 
           we will look at $[x_1x_3, x_1x_3m_3]$ and the other cases are similar.
           If $x_1,x_2^2$ both can not divide $m_3$, then 
           there exist $4\leq i_1 < \cdots < i_5$ and $m_3'$
           such that $m_3=x_{i_1}\cdots x_{i_5}m_3'$, so that 
           $x_1x_3x_{i_1}, \ldots, x_1x_3x_{i_5}$ are 
           $5$ degree-$3$ monomials in this interval; 
           if $x_1$ divides $m_3$ and $x_2^2$ can not divide $m_3$, then 
           there exist $4\leq i_1 < \cdots < i_4$ and $m_3'$
           such that $m_3=x_1x_{i_1}\cdots x_{i_4}m_3'$, so that 
           $x_1^2x_3,x_1x_3x_{i_1}, \ldots, x_1x_3x_{i_4}$ are 
           $5$ degree-$3$ monomials in this interval;  
           if $x_1$ can not divide $m_3$ and $x_2^2$ divides $m_3$, then 
           there exist $4\leq i_1 < \cdots < i_4$ and $m_3'$
           such that $m_3=x_2^2x_{i_1}\cdots x_{i_4}m_3'$, so that 
           $x_1x_2^2,x_1x_3x_{i_1}, \ldots, x_1x_3x_{i_4}$ are 
           $5$ degree-$3$ monomials in this interval; 
           if $x_1, x_2^2$ both divides $m_3$, then 
           there exist $4\leq i_1 < i_2 < i_3$ and $m_3'$
           such that $m_3=x_1x_2^2x_{i_1}x_{i_2} x_{i_3}m_3'$, so that 
           $x_1^2x_2,x_1x_2^2,x_1x_3x_{i_1}, x_1x_3x_{i_2}, x_1x_3x_{i_3}$ are 
           $5$ degree-$3$ monomials in this interval;          
\end{itemize}
Therefore, all these $10$ disjoint intervals will have at leat $50$
degree-$3$ monomials in total. However, the last $15$ monomials of degree $3$ 
in $P_{L_2}^g$ are $x_2^2x_3, \ldots, x_2x_3x_{10}$, which can not belong to 
these $10$ intervals, so that there are only $46$ degree-$3$ monomials left
for these $10$ intervals. $46<50$, and we have a contradiction. 
So, the assumption $\mbox{sdepth}(L_2)=6$ is not true, 
and then $\mbox{sdepth}(L_2)\leq 5$.

Let $L_3=(x_1^2, x_1x_2, \ldots, x_1x_{10})\subset S$ then $L_2=L_3+(x_2^2, x_2x_3)$.
It is esay to see that 
$\mbox{sdepth}(L_3)=\mbox{sdepth}(x_1, x_2, \ldots, x_{10})=\lceil \frac{10}{2} \rceil=5$.
Hence, $L_3$ has a Stanley decomposition $\mathcal{D}_1$ such that $\mbox{sdepth}(\mathcal{D}_1)=5$.
Let 
\[
\mathcal{D}_2=\mathcal{D}_1\bigoplus x_2^2K[x_2, x_3, \ldots, x_{10}] \bigoplus x_2x_3K[x_3, x_4, \ldots, x_{10}].
\]
We can check that $\mathcal{D}_2$ is a Stanley decomposition of $L_2$,
so that $\mbox{sdepth}(L_2)\geq \mbox{sdepth}(\mathcal{D}_2)=5$. 
Therefore, $\mbox{sdepth}(L_2)=5<\mbox{hdepth}(L_2)=6$.
\end{counterexample}

\begin{counterexample}\label{C:cex4}
Consider $L_2^\sigma=(x_1x_2, x_1x_3, \ldots, x_1x_{11}, x_2x_3, x_2x_4)$. 
It is a squarefree lex ideal in $S'=K[x_1, \ldots, x_{11}]$. 
From Counterexample \ref{C:cex2},We already know that $\mbox{hdepth}(L_2^\sigma)=7$.

To compute $\mbox{sdepth}(L_2^\sigma)$, we use the method of \cite{B:HVZ} and let $g'=(1,1,\ldots,1)$. 
Instead of thinking of $P_{L_2^\sigma}^{g'}$ as a subset of $\mathbb{N}^{11}$, 
we will equivalently view $P_{L_2^\sigma}^{g'}$ as the set of all monomials 
which divides $x_1x_2\cdots x_{11}$ and can be divided by one of the generators of $L_2^\sigma$. 

All the degree $2$ monomials in $P_{L_2^\sigma}^{g'}$ are
\[
x_1x_2,x_1x_3,x_1x_4,x_1x_5,x_1x_6,x_1x_7,x_1x_8,x_1x_9,x_1x_{10},x_1x_{11}, x_2x_3 , x_2x_4.
\]
All the degree $3$ monomials in $P_{L_2^\sigma}^{g'}$ are
\[
\begin{matrix}
x_1x_2x_3&x_1x_2x_4&x_1x_2x_5&x_1x_2x_6&x_1x_2x_7&x_1x_2x_8&x_1x_2x_9&x_1x_2x_{10}\\
x_1x_2x_{11}&x_1x_3x_4&x_1x_3x_5&x_1x_3x_6&x_1x_3x_7&x_1x_3x_8&x_1x_3x_9&x_1x_3x_{10}\\
x_1x_3x_{11}&x_1x_4x_5&x_1x_4x_6&x_1x_4x_7&x_1x_4x_8&x_1x_4x_9&x_1x_4x_{10}&x_1x_4x_{11}\\
x_1x_5x_6&x_1x_5x_7&x_1x_5x_8&x_1x_5x_9&x_1x_5x_{10}&x_1x_5x_{11}&x_1x_6x_7&x_1x_6x_8\\
x_1x_6x_9&x_1x_6x_{10}&x_1x_6x_{11}&x_1x_7x_8&x_1x_7x_9&x_1x_7x_{10}&x_1x_7x_{11}&x_1x_8x_9\\
x_1x_8x_{10}&x_1x_8x_{11}&x_1x_9x_{10}&x_1x_9x_{11}&x_1x_{10}x_{11} &\ &\ &\ \\
x_2x_3x_4&x_2x_3x_5&x_2x_3x_6&x_2x_3x_7&x_2x_3x_8&x_2x_3x_9&x_2x_3x_{10}&x_2x_3x_{11}\\
x_2x_4x_5&x_2x_4x_6&x_2x_4x_7&x_2x_4x_8&x_2x_4x_9&x_2x_4x_{10}&x_2x_4x_{11} &\ .
\end{matrix}
\]
There are $12$ degree-$2$ monomials and $60$ degree-$3$ monomials in $P_{L_2^\sigma}^{g'}$.

Assume that $\mbox{sdepth}(L_2^\sigma)=7$. Then there exists a partition $\mathcal{P}$
of $P_{L_2^\sigma}^{g'}$ such that the Stanley decomposition induced by $\mathcal{P}$ 
has Stanley depth $7$. Suppose that in the partition $\mathcal{P}$ we have 
intervals $[x_1x_2, x_1x_2m_2]$, $[x_1x_3, x_1x_3m_3]$, 
$\ldots$, $[x_1x_{11}, x_1x_{11}m_{11}]$. It is easy to see that each of these $10$
intervals has at least $5$ degree-$3$ monomials. 
Indeed, for the interval $[x_1x_2,x_1x_2m_2]$, 
since $\mbox{sdepth}(L_2^\sigma)=7$, it follows that there exist
$3\leq i_1<\ldots<i_5$ and $m_2'$ such that $m_2=x_{i_1}\cdots x_{i_5}m_2'$,
so that $x_1x_2x_{i_1}, \ldots, x_1x_2x_{i_5}$ are 
$5$ degree-$3$ monomials in this interval; 
for the other $9$ intervals, the argument is similar. 

Therefore, all these $10$ disjoint intervals will have at leat $50$
degree-$3$ monomials in total. However, the last $15$ monomials of degree $3$ 
in $P_{L_2^\sigma}^{g'}$ are $x_2x_3x_4, \ldots, x_2x_4x_{11}$, which can not belong to 
these $10$ intervals, so that there are only $45$ degree-$3$ monomials left
for these $10$ intervals. $45<50$, and we have a contradiction. 
So, the assumption $\mbox{sdepth}(L_2^\sigma)=7$ is not true, 
and then $\mbox{sdepth}(L_2^\sigma)\leq 6$.

Let $L_4=( x_1x_2, x_1x_3, \ldots, x_1x_{11})\in S'$ then $L_2^\sigma=L_4+(x_2x_3, x_2x_4)$.
It is esay to see that 
$\mbox{sdepth}(L_4)=\mbox{sdepth}(x_2, x_3, \ldots, x_{11})=11-\lfloor \frac{10}{2} \rfloor=6$,
where the second identity is by the result of \cite{B:Sh1}.
Hence, $L_4$ has a Stanley decomposition $\mathcal{D}_3$ such that $\mbox{sdepth}(\mathcal{D}_3)=6$.
Let 
\[
\mathcal{D}_4=\mathcal{D}_3\bigoplus x_2x_3K[x_2, x_3, x_4, \ldots, x_{11}] \bigoplus x_2x_4K[x_2, x_4, x_5\ldots, x_{11}].
\]
We can check that $\mathcal{D}_4$ is a Stanley decomposition of $L_2^\sigma$,
so that $\mbox{sdepth}(L_2^\sigma)\geq \mbox{sdepth}(\mathcal{D}_4)=6$. 
Therefore, $\mbox{sdepth}(L_2^\sigma)=6<\mbox{hdepth}(L_2^\sigma)=7$.
\end{counterexample}

\begin{remark}
From the above two counterexamples, we see that $\mbox{sdepth}(L_2)$ and $\mbox{sdepth}(L_2^\sigma)$
satisfies the identity $\mbox{sdepth}(I)=\mbox{sdepth}(I^\sigma)-(m-n)$, where $I$ is a strongly stable 
monomial ideal generated by monomials of the same degree. 
As mentioned in \cite{B:Sh2}, this identity was first suggested by Herzog. 
If we can prove this identity for lex ideals 
generated by monomials of the same degree, 
then we can reduce the study of $\mbox{sdepth}(\mathfrak{m}^d)$ to 
the study of $\mbox{sdepth}(I_{n+d-1,d})$, which would be a big progress 
in the study of Stanley depth.

From Counterexample \ref{C:cex4}, the author feels that 
$\mbox{sdepth}(I_{n,d})=d+\lfloor \frac{n-d}{d+1} \rfloor$ probably
does not hold for all $n\geq 5d+4$. The simplest case is $I_{14,2}$. 
We know $\mbox{hdepth}(I_{14,2})=2+\lfloor \frac{12}{3} \rfloor=6$,
and we can use Algorithm \ref{A:squarefree} to find a Hilbert decomposition 
of $I_{14,2}$ whose Hilbert depth is  $6$. 
It would be interesting to figure out if there exists a Stanley decomposition 
of  $I_{14,2}$ whose Stanley depth is  $6$.
\end{remark}


\begin{thebibliography}{BHTT}
      \bibitem[AHH]{B:AHH}
      A. Aramova, J. Herzog, T. Hibi: 
      Shifting operations and graded Betti numbers, 
      \emph{J. Algebraic Combin. } 12 (2000),  207--222.
      \bibitem[BHK]{B:BHK}
      C. Bir\'{o}, D.M. Howard, M.T. Keller, W.T. Trotter, S.J. Young: 
      Interval partitions and Stanley depth, 
      \emph{J. Combin. Theory Ser. A} 117 (2010),  475--482.
      \bibitem[BKU1]{B:BKU1}
      W. Bruns, C. Krattenthaler, J. Uliczka: 
      Stanley decompositions and Hilbert depth in the Koszul complex, 
      \emph{J. Commut. Algebra} 2 (2010),  327--357.  
      \bibitem[BKU2]{B:BKU2}
      W. Bruns, C. Krattenthaler, J. Uliczka: 
      Hilbert depth of powers of the maximal ideal, 
      \emph{Commutative algebra and its connections to geometry,}  (2011),  pp. 1--12.  
      \bibitem[BMU]{B:BMU}
      W. Bruns, J. J. Moyano-Fern\'{a}ndez, J. Uliczka: 
      Hilbert regularity of $\mathbb{Z}$-graded modules over polynomial rings, 
      Preprint, arXiv:1308.2917v1 (2013).      
      \bibitem[Ci]{B:Ci}
      M. Cimpoeas: 
      Stanley depth of complete intersection monomial ideals, 
      \emph{Bull. Math. Soc. Sci. Math. Roumanie (N.S.)} 51(99) (2008),  205--211. 
      \bibitem[EK]{B:EK}
      S. Eliahou, M. Kervaire: 
      Minimal resolutions of some monomial ideals,
      \emph{J. Algebra} 129, (1990), 1--25.
      \bibitem[GLW]{B:GLW}
      M. Ge, J. Lin, Y. Wang: 
      Hilbert series and Hilbert depth of squarefree Veronese ideals, 
      \emph{J. Algebra} 344 (2011),  260--267.        
      \bibitem[HVZ]{B:HVZ}
      J. Herzog, M. Vladoiu, X. Zheng: 
      How to compute the Stanley depth of a monomial ideal, 
      \emph{J. Algebra} 322 (2009),  3151--3169.                  
      \bibitem[KSS]{B:KSS}
      M.T. Keller, Y.-H. Shen, N. Streib, S.J. Young: 
      On the Stanley depth of squarefree Veronese ideals, 
      \emph{J. Algebraic Combin. } 33 (2011),  313--324. 
      \bibitem[Ma]{B:Ma}
      F. Macaulay: 
      Some properties of enumeration in the theory of modular systems, 
      \emph{Proc. London Math. Soc.} 26 (1927),  531--555. 
      \bibitem[Po]{B:Po}
      A. Popescu: 
      An algorithm to compute the Hilbert depth, 
      Preprint, arXiv:1307.6084v2 (2013).          
      \bibitem[Sh1]{B:Sh1}
      Y.-H. Shen: 
      Stanley depth of complete intersection monomial ideals and upper-discrete partitions, 
      \emph{J. Algebra} 321 (2009),  1285--1292.                   
      \bibitem[Sh2]{B:Sh2}
      Y.-H. Shen: 
      Lexsegment ideals of Hilbert depth 1, 
      Preprint, arXiv:1208.1822 (2012).  
      \bibitem[St]{B:St}
      R.P. Stanley: 
      Linear Diophantine equations and local cohomology, 
      \emph{Invent. Math.} 68 (1982),  175--193.   
      \bibitem[Ul]{B:Ul}
      J. Uliczka: 
      Remarks on Hilbert series of graded modules over polynomial rings, 
      \emph{Manuscripta Math.} 132 (2010),  159--168.               
      
\end{thebibliography}
\end{document}